\documentclass{amsart}

\usepackage{hyperref}
\usepackage{amsmath}
\usepackage{amssymb}
\usepackage{amscd}
\usepackage{graphicx}
\usepackage{mathdots}
\usepackage{gensymb}
\usepackage{epstopdf}

\newtheorem{thm}{Theorem}%[section]
\newtheorem{prop}[thm]{Proposition}

\theoremstyle{definition}

\theoremstyle{remark}

\numberwithin{equation}{section}

\newcommand{\group}[1]{\mathfrak{#1}}

\newcommand{\Aut}{\operatorname{Aut}}
\newcommand{\hur}{H}

\newcommand{\mon}{\vec{H}}

\newcommand{\id}{\operatorname{id}}

\newcommand{\Walk}{\operatorname{Walk}}

\begin{document}

%%
%% The title of the paper goes here.  Edit to your title.
%%

\title[Monotone Hurwitz Generating Functions]
{On the Convergence of Monotone Hurwitz Generating Functions}
%%
%% Now edit the following to give your name and address:
%% 

\author{I. P. Goulden}
\address{Department of Combinatorics \& Optimization,
University of Waterloo, Canada}
\email{ipgoulden@uwaterloo.ca}

\author{Mathieu Guay-Paquet}
\address{LACIM, Universit\'e du Qu\'ebec \`a Montr\'eal, Canada}
\email{mathieu.guaypaquet@lacim.ca}

\author{Jonathan Novak}
\address{Department of Mathematics, University of California, San Diego}
\email{jinovak@ucsd.edu}

%%%
%%% The following is for the abstract.  The abstract is optional and
%%% if not used just delete, or comment out, the following.
%%%

\begin{abstract}
	Monotone Hurwitz numbers were introduced by the 
	authors as a combinatorially natural desymmetrization of the 
	Hurwitz numbers studied in enumerative algebraic
	geometry.  Over the course of several papers, we developed the
	structural theory of monotone Hurwitz numbers and demonstrated
	that it is in many ways parallel to that of their classical counterparts. 
	In this note, we identify an 
	important difference between the monotone and classical worlds:
	fixed-genus generating functions for monotone double Hurwitz
	numbers are absolutely summable, whereas those for classical
	double Hurwitz numbers are not.  This property is crucial for 
	applications of monotone Hurwitz theory in analysis.
	We quantify the growth rate of monotone Hurwitz numbers in fixed
	genus by giving universal upper and lower bounds on the radii of 
	convergence of their generating functions.
\end{abstract}

%%
%%  LaTeX will not make the title for the paper unless told to do so.
%%  This is done by uncommenting the following.
%%
\maketitle

%%
%% LaTeX can automatically make a table of contents.  This is done by
%% uncommenting the following:
%%
%\tableofcontents

%%
%%  To enter text is easy.  Just type it.  A blank line starts a new
%%  paragraph. 
%%

\section{Introduction}
Let us identify the symmetric group $\group{S}(d)$ with its right Cayley graph, as generated
by the conjugacy class of transpositions.  Furthermore, let us equip $\group{S}(d)$ with the 
edge-labelling in which each edge corresponding to the transposition 
$(s\ t)$ is marked with $t$, the larger of the two numbers interchanged.  
A restricted version of this labelling was introduced by Stanley \cite{Stanley} as an EL-labelling of the 
noncrossing partition lattice $\operatorname{NC}(d)$ related to parking functions;
it was subsequently ported to the symmetric group by Biane \cite{Biane} using a 
natural embedding $\operatorname{NC}(d) \rightarrow \group{S}(d)$.
Figure \ref{fig:Cayley} shows $\group{S}(4)$ together with the Biane--Stanley
edge labelling: $2$-edges are drawn in blue, $3$-edges in yellow, and $4$-edges in red.
In general, we will refer to the edge labels of $\group{S}(d)$ as \emph{colours}.

\begin{figure}
	\includegraphics{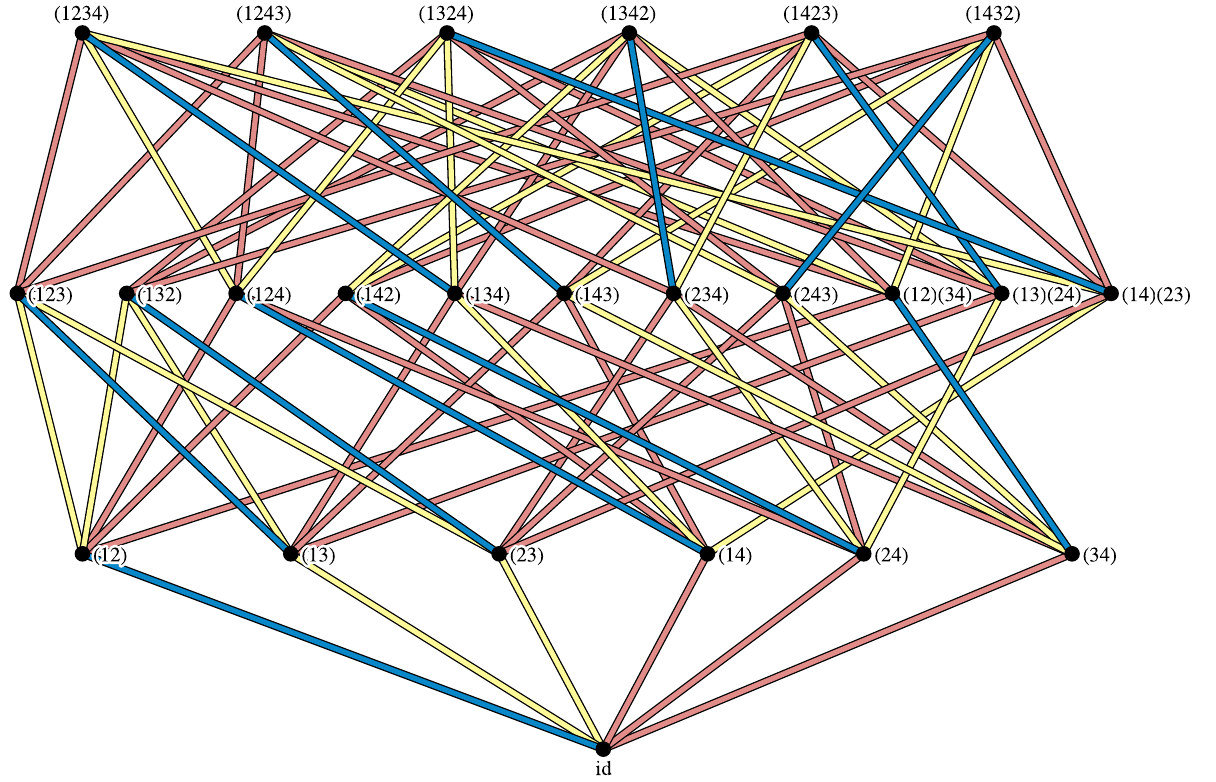}
	\caption{\label{fig:Cayley} $\group{S}(4)$ with the Biane--Stanley
	edge-labelling (figure by M. LaCroix).}
\end{figure}

A walk on $\group{S}(d)$ is said to be \emph{monotone} if the colours of the edges it traverses
form a weakly increasing sequence.  Once this notion has been introduced, a natural
question is to count monotone walks of a given length between two given permutations.  
It turns out that this question, which is of a purely combinatorial nature,
is closely connected to certain problems in analysis.  
In particular, it is known \cite{Novak:BCP} that, for any 
$\rho,\sigma \in \group{S}(d)$ and any $N \geq d$, one has 

	\begin{equation}
		\label{eqn:Weingarten}
		\int_{U(N)} u_{11} \dots u_{dd} 
		\overline{u_{1\rho^{-1}\sigma(1)} \dots u_{d\rho^{-1}\sigma(d)}} \mathrm{d}U
		= \frac{1}{N^d} \sum_{r=0}^{\infty} (-1)^r \frac{\vec{w}^r(\rho,\sigma)}{N^r},
	\end{equation}
	
\noindent
where the integration is against the Haar probability measure on the compact 
group of $N \times N$ complex unitary matrices $U=[u_{ij}]_{i,j=1}^N$, and 
$\vec{w}^r(\rho,\sigma)$ is the number of $r$-step monotone walks on $\group{S}(d)$
from $\rho$ to $\sigma$.  It is a non-obvious fact that the number $\vec{w}^r(\rho,\sigma)$,
and hence the integral in \eqref{eqn:Weingarten}, depends on the permutations
$\rho$ and $\sigma$ only through the cycle type of $\rho^{-1}\sigma$.  
Moreover, as explained in \cite{MN}, 
the formula \eqref{eqn:Weingarten} implies that the computation
of inner products of homogeneous polynomial functions of
degree $d \leq N$ in the Hilbert space $L^2(U(N),\operatorname{Haar})$
is equivalent to the enumeration of monotone walks on $\group{S}(d)$.  The computation 
of such inner products is important in various parts of mathematical physics, including
lattice quantum chromodynamics \cite{DH,Samuel} and string theory \cite{GT}, 
and from this perspective 
\eqref{eqn:Weingarten} serves as a kind of Feynman diagram expansion.
Polynomial integrals on $U(N)$ are also ubiquitous in random matrix
theory, where in the course of various moment computations it is necessary to ``integrate 
out'' eigenvector data to get at the eigenvalue distribution of a unitarily invariant random matrix.

Another perspective on the problem of enumerating monotone walks on $\group{S}(d)$
is to consider it as a desymmetrization of the \emph{double Hurwitz problem} from enumerative
algebraic geometry.  Given two partitions $\alpha,\beta \vdash d$, and a nonnegative integer
$g \geq 0$, let $\hur_g(\alpha,\beta)$ denote the total number of $(2g-2+\ell(\alpha)+\ell(\beta))$-step
walks on $\group{S}(d)$, not necessarily monotone,
which begin at a permutation of cycle type $\alpha$, end
at a permutation of cycle type $\beta$, and have the property that their steps
and endpoints together generate a transitive subgroup of $\group{S}(d)$.
The numbers $\hur_g(\alpha,\beta)$
are known as the \emph{double Hurwitz numbers} because, 
by a classical construction due to Hurwitz \cite{Hurwitz}, $\frac{1}{d!}\hur_g(\alpha,\beta)$
is a weighted count of degree $d$ branched covers of the Riemann sphere $\mathbf{P}^1$ by a
genus $g$ compact, connected Riemann surface which have ramification profile $\alpha$ over $\infty$, 
$\beta$ over $0$, and the simplest non-trivial branching over $2g-2+\ell(\alpha)+\ell(\beta)$
additional specified points of $\mathbf{P}^1$, the number of which is determined by the 
Riemann--Hurwitz formula.  It is convenient to introduce the notation 

	\begin{equation*}
		r_g(\alpha,\beta) = 2g-2+\ell(\alpha)+\ell(\beta)
	\end{equation*}
	
\noindent
for this number.

The double Hurwitz numbers have been objects of intense study in 
both algebraic combinatorics and algebraic geometry, see e.g.~\cite{GJV,OP1,OP2} and references
therein, and it is known that they exhibit a wide array of subtle structural properties.
In \cite{GGN1,GGN2,GGN3}, we introduced and studied the \emph{monotone double Hurwitz numbers}
$\mon_g(\alpha,\beta)$, which count the same walks as the double Hurwitz numbers
$\hur_g(\alpha,\beta)$, but with the monotonicity constraint imposed.  In \cite{GGN1,GGN2},
we demonstrated that the monotone double Hurwitz numbers enjoy a high degree of structural
similarity with the classical double Hurwitz numbers, and in \cite{GGN3} we applied this
structure to the asymptotic analysis of unitary matrix integrals via \eqref{eqn:Weingarten}.
While our analysis was purely combinatorial,
geometric explanations of the concordance between monotone and 
classical Hurwitz theory are gradually appearing, see e.g. \cite{DDM,DK}.

The present note is not about structural analogies between monotone and classical Hurwitz numbers,
but rather about an important quantitative difference between the two: growth rate.
Obviously, the monotone Hurwitz number $\mon_g(\alpha,\beta)$ is smaller
than its classical counterpart $\hur_g(\alpha,\beta)$ --- but how much smaller?
While explicit formulas for Hurwitz numbers are few and far between (see the 
introduction of \cite{GGN2} for a summary of known formulas), 
the elementary genus zero formulas

	\begin{equation}
		\label{eqn:ElementaryFormulas}
		\frac{1}{d!}\mon_0(1^d,d) = \frac{1}{d^2} \binom{2d-2}{d-1},
		\qquad \frac{1}{d!}\hur_0(1^d,d) = d^{d-3},
	\end{equation}
	
\noindent
already reveal a stark distinction: the classical
Hurwitz number grows superexponentially in the degree $d$, 
whereas its monotone analogue exhibits only exponential growth.  
This sharp drop in growth rate plays an important
role in analytic applications of monotone Hurwitz theory \cite{GGN3,Novak:JSP},
where one needs to know that, for any fixed $g \geq 0$, the generating function

	\begin{equation}
		\label{eqn:FixedGenus}
		\vec{\mathbf{F}}_g(z) 
		= \sum_{d=1}^\infty \left( \sum_{\alpha,\beta \vdash d} \mon_g(\alpha,\beta)\right)
		\frac{z^d}{d!}
	\end{equation}
	
\noindent
enumerating all genus $g$ monotone transitive walks on $\group{S}(d)$
is absolutely summable.  In fact, one needs the stronger statement that the radius of convergence
of $\vec{\mathbf{F}}_g(z)$ is bounded below by a positive universal constant, so that one may consider 
$\{\vec{\mathbf{F}}_g(z)\}$ as a family of 
holomorphic functions defined on a common complex domain.  The goal of this note is to give a 
detailed proof of a precise form of this growth property.

	\begin{thm}
	\label{thm:Main}
		For any $g \geq 0$, the radius of convergence of the generating function
		$\vec{\mathbf{F}}_g(z)$ is at least $1/54$ and at most $2/27$.
	\end{thm}

The proof of Theorem \ref{thm:Main} consists of two components.  

The first component,
which was established in \cite{GGN3} using the results of \cite{GGN1,GGN2}, is
concerns the generating function 

	\begin{equation*}
		\vec{\mathbf{S}}_g(z) = \sum_{d=1}^{\infty} \mon_g(1^d,1^d)\frac{z^d}{d!}
	\end{equation*}	
	
\noindent
for genus $g$ monotone \emph{simple} Hurwitz numbers $\mon_g(1^d,1^d)$, 
which count monotone transitive identity-based 
loops on the symmetric groups of length $r_g(1^d,1^d)$.  More precisely,
in \cite[Theorem 4.1]{GGN3}, we proved that $\vec{\mathbf{S}}_g(z)$ has 
radius of convergence $2/27$ by expressing it in terms of the 
Gauss hypergeometric function

	\begin{equation}
		\label{eqn:Hypergeometric}
		{}_2F_1\Big(\frac{2}{3},\frac{4}{3},\frac{3}{2};\frac{27}{2}z\Big).
	\end{equation}
	
\noindent
This yields an upper bound of $2/27$ on the radius of convergence of 
$\vec{\mathbf{F}}_g(z)$.

The second component of the proof of Theorem \ref{thm:Main},
which we establish in Proposition \ref{prop:Inequalities} of this paper, 
is the lower bound in Theorem \ref{thm:Main}.
This is obtained by showing that the terms of $\vec{\mathbf{F}}_g(z)$
are larger than those of $\vec{\mathbf{S}}_g(z)$ by at most an explicit
exponential factor:

	\begin{equation}
		\label{eqn:Inequality}
		\sum_{\alpha,\beta \vdash d} \mon_g(\alpha,\beta) \leq 4^d \mon_g(1^d,1^d).
	\end{equation}

\noindent
To prove the inequality \eqref{eqn:Inequality}, we make use of a
\emph{sorting action} of the symmetric group
$\group{S}(r)$ on $r$-step walks on $\group{S}(d)$ which is in fact a close relative
of Hurwitz's classical braid group action on branched coverings (see e.g.~\cite{EEHS} for 
an informative discussion of braids and branched coverings).  

Although the inequality \eqref{eqn:Inequality} is sufficient for our purposes,
it may not be sharp.  In fact, we conjecture that the radius of convergence of 
$\vec{\mathbf{F}}_g(z)$ 
is exactly $2/27$ for all $g \geq 0$, so that, by Pringsheim's theorem, these generating
functions have a common dominant singularity at the critical point $z_c=2/27$.  This 
would be consistent with various similarities between Hurwitz numbers, maps on 
surfaces, and matrix models, where common singular 
behaviour of generating functions across all genera
is an established phenomenon, see e.g.~\cite{OP1} and references therein.

Finally, we cannot resist mentioning the striking fact, pointed out to us by 
P. Di Francesco, that the peculiar number $2/27$ also occurs in the enumeration of finite 
groups.  Namely, given a prime $p$, let $\Gamma_p(N)$ 
denote the number of isomorphism classes of groups
of order $p^N$; it is known (see \cite{Pyber} and references therein) that 

	\begin{equation}
		\lim_{N \rightarrow \infty} \frac{1}{N^3} 
		\log\Gamma_p(N) = \frac{2}{27} \log p.
	\end{equation}
	
\noindent
Any explanation of this coincidence would doubtless be very interesting.

\section{Proof of Theorem \ref{thm:Main}}
\label{sec:Proof}

	\subsection{Swapping steps}
	Fix permutations $\rho,\sigma \in \group{S}(d)$ and a positive integer $r$,
	and consider the set $\Walk_r(\rho,\sigma)$ of all (not necessarily 
	monotone or transitive) $r$-step walks from $\rho$
	to $\sigma$ (assume $r$ is large enough so that this set is non-empty).
	Equivalently, $\Walk_r(\rho,\sigma)$ is the set of $r$-tuples
	
		\begin{equation}
			\label{eqn:Walk}
			W=((s_1\ t_1), \dots, (s_r\ t_r)), \quad s_i < t_i,
		\end{equation}
		
	\noindent
	of transpositions such that 
	
		\begin{equation}
			\sigma = \rho(s_1\ t_1) \dots (s_r\ t_r).
		\end{equation}
		
	\noindent
	Given a walk $W$ as in \eqref{eqn:Walk}, we refer to the list
	$(t_1,\dots,t_r)$ as the \emph{spectrum} of $W$, and the elements
	of this list as \emph{colours}.  Let 
	
		\begin{equation}
			\chi: \Walk_r(\rho,\sigma) \longrightarrow \{2,\dots,d\}^r
		\end{equation}
		
	\noindent
	be the map which sends a walk $W$ to its spectrum $\chi(W)$. 
	
	The symmetric group $\group{S}(r)$ acts naturally on spectra
	$\chi \in  \{2,\dots,d\}^r$ by permuting colours: for each $\pi \in 
	\group{S}(r)$ we have a corresponding operator
		
		\begin{equation*}
			S_\pi(t_1,\dots,t_r) = (t_{\pi(1)},\dots,t_{\pi(r)})
		\end{equation*}
		
	\noindent
	on spectra, and it is clear that $S_{\pi_1}S_{\pi_2}=S_{\pi_1\pi_2}$.
		
	One may also define an action of $\group{S}(r)$ on  
	$\Walk_r(\rho,\sigma)$.  To define this action, we begin by 
	introducing operators $R_i$, $1 \leq i \leq r-1$ which 
	act on walks as follows: we set
			
		\begin{equation*}
			R_i(s_1\ t_1) \dots (s_i\ t_i)(s_{i+1}\ t_{i+1}) \dots (s_r\ t_r)
			= (s_1\ t_1) \dots (*\ t_{i+1})(*\ t_i) \dots (s_r\ t_r),
		\end{equation*}
		
	\noindent
	where 
	
		\begin{equation*}
			(*\ t_{i+1}) = (s_i\ t_i)(s_{i+1}\ t_{i+1})(s_i\ t_i) \quad\text{ and }\quad
			(*\ t_i) = (s_i\ t_i)
		\end{equation*}
		
	\noindent 
	if $t_i < t_{i+1}$,
	
		\begin{equation*}
			(*\ t_{i+1}) = (s_{i+1}\ t_{i+1}) \quad \text{ and } \quad
			(*\ t_i) = (s_{i+1}\ t_{i+1})(s_i\ t_i)(s_{i+1}\ t_{i+1}) 
		\end{equation*}
		
	\noindent
	if $t_i > t_{i+1}$, and finally
	
		\begin{equation*}
			(*\ t_{i+1}) = (s_{i+1}\ t_{i+1}) \quad\text{ and }\quad 
			(*\ t_i) = (s_i\ t_i)
		\end{equation*}
		
	\noindent
	if $t_i=t_{i+1}$.  In words, the effect of $R_i$ is to swap
	the $i$ and $(i+1)$st steps of $W$ and then conjugate
	the step with the higher colour by the step with the lower colour; 
	if the two steps are the same colour, $R_i$ does nothing.
	
	\begin{prop}
		\label{prop:Coxeter}
		The operators $R_i$ satisfy the Coxeter relations:
		
			\begin{align*}
				R_i^2 &= I, \quad 1 \leq i \leq r-1 \\
				R_iR_{i+1}R_i &= R_{i+1}R_iR_{i+1}, 
					\quad 1 \leq i \leq r-2 \\
				R_iR_j &= R_jR_i, \quad 1 \leq i,j \leq r-1,\ |i-j| \geq 2.
			\end{align*}
		
		\noindent
		Moreover, if $W$ is transitive, so is $R_iW$.
	\end{prop}
	
	\begin{proof}
		Straightforward case-checking.
	\end{proof}
	
	\noindent
	In view of Proposition \ref{prop:Coxeter}, we may define a group homomorphism
	
		\begin{equation*}
			R:\group{S}(r) \longrightarrow \Aut \Walk_r(\rho,\sigma)
		\end{equation*}
		
	\noindent
	by mapping the Coxeter generator $(i\ i+1)$ to the operator $R_i$.
	We conclude that $(\Walk_r(\rho,\sigma),R)$ and $(\{2,\dots,d\}^r,S)$ are 
	permutation representations of $\group{S}(r)$, and that the diagram
	
		\begin{equation*}
		\begin{CD}
			\Walk_r(\rho,\sigma) @>\chi>> \{2,\dots,d\}^r \\
			@VV{R_\pi}V @VV{S_\pi}V \\
			\Walk_r(\rho,\sigma) @>\chi>> \{2,\dots,d\}^r
		\end{CD}	
		\end{equation*}	
		
	\noindent
	commutes for all $\pi \in \group{S}(r)$.
	
	\begin{prop}
		The restriction of $\chi$ to any single orbit of $(\Walk_r(\rho,\sigma),R)$ 
		is injective, with image the corresponding orbit of $(\{2,\dots,d\}^r,S)$.
	\end{prop}
	
	\begin{proof}
		Since $\chi$ respects the action of $\group{S}(r)$, it sends each orbit of $(\Walk_r(\rho,\sigma),R)$ to an orbit of $(\{2,\dots,d\}^r,S)$; the question is whether the image orbit is smaller than the source orbit, and we can answer it by looking at stabilizers.
		
		Consider the element of the image orbit which is sorted, that is, the spectrum $(t_1, \dots, t_r)$ such that $t_1 \leq \dots \leq t_r$. If the colour 2 appears $m_2$ times in the spectrum, the colour 3 appears $m_3$ times in the spectrum, and so on up to $m_d$, then the stabilizer of the spectrum is the Young subgroup
		
		\begin{equation}
			\label{eqn:stabilizer}
			\group{S}(m_2) \times \group{S}(m_3) \times \dots \times \group{S}(m_d) \subseteq \group{S}(r).
		\end{equation}
		
		\noindent
		This subgroup is generated by the transpositions $(i\ i+1)$ which exchange two adjacent copies of the same colour in the spectrum.
		
		Now, consider a walk
		
		\begin{equation*}
			W=((s_1\ t_1), \dots, (s_r\ t_r))
		\end{equation*}
		
		\noindent
		with spectrum $(t_1, \dots, t_r)$. By definition, we have $R_i W = W$ when $t_i = t_{i+1}$, so the stabilizer of $W$ contains the Young subgroup \eqref{eqn:stabilizer}.
		Since $\chi$ is well-defined, the stabilizer of $W$ can be no bigger, so in fact it is exactly this Young subgroup.
		By the orbit-stabilizer theorem, the two orbits have the same size, so $\chi$ is a bijective correspondence between them, as needed.
		
		As a further consequence, note that each orbit of $(\Walk_r(\rho,\sigma),R)$ contains a unique monotone walk, since each orbit of $(\{2,\dots,d\}^r,S)$ contains a unique sorted spectrum.
	\end{proof}
	
	\subsection{Sorting colours}
	For any $c \in \{1,\dots,d-1\}$, define $\mon_g(\alpha,\beta;c)$ to be
	the number of walks counted by $\mon_g(\alpha,\beta)$ whose spectrum
	contains exactly $c$ distinct colours. 
	
	\begin{prop}
		\label{prop:Inequalities}
		For any $d \geq 1$ and $g \geq 0$, we have
		
			\begin{equation*}
				\sum_{c=1}^{d-1} 3^c \mon_g(1^d,1^d;c)
				\leq \sum_{\alpha,\beta \vdash d} \mon_g(\alpha,\beta)
				\leq \sum_{c=1}^{d-1} 4^c \mon_g(1^d,1^d;c).
			\end{equation*}
	\end{prop}
	
	\begin{proof}
		The monotone double Hurwitz number $\mon_g(\alpha,\beta)$
		counts monotone transitive walks 
		
			\begin{equation}
				\label{eqn:MonotoneWalk}
				\sigma = \rho(s_1\ t_1) \dots (s_{r_g(\alpha,\beta)}\ t_{r_g(\alpha,\beta)}),
				\quad t_1< \dots < t_{r_g(\alpha,\beta)}
			\end{equation}
			
		\noindent
		on $\group{S}(d)$
		which begin at a permutation $\rho$ of cycle type $\alpha$, 
		end at a permutation $\sigma$ of cycle type $\beta$, and take 
		a total of $r_g(\alpha,\beta)=2g-2+\ell(\alpha)+\ell(\beta)$ steps.  Given such 
		a walk, there is a unique \emph{strictly} monotone walk
		
			\begin{equation*}
				\rho = \id (s_1^\rho\ t_1^\rho) \dots (s_{r_0(\alpha)}^\rho\ t_{r_0(\alpha)}^\rho),
				\quad t_1^\rho <  \dots < t_{r_0(\alpha)}^\rho,
			\end{equation*}
			
		\noindent
		from the identity permutation to $\rho$,
		and this walk is of length $r_0(\alpha) = d- \ell(\alpha)$.  
		This fact, in different language, is due independently to 
		Jucys \cite{Jucys}, and Diaconis and Greene \cite{DG}; see \cite{MN}
		for a proof.
		Similarly, there is a unique strictly monotone walk
		
			\begin{equation*}
				\sigma = 
				\id (s_1^\sigma\ t_1^\sigma) \dots (s_{r_0(\beta)}^\sigma\ t_{r_0(\beta)}^\sigma),
				\quad t_1^\sigma <  \dots < t_{r_0(\beta)}^\sigma,
			\end{equation*}
		
		\noindent	
		from the identity to $\sigma$, and this walk has length 
		$r_0(\beta)=d-\ell(\beta)$.  Thus the walk \eqref{eqn:MonotoneWalk}
		gives rise to a transitive, identity based loop
		
			\begin{equation}
				\label{eqn:TransitiveLoop}
				\id = \id(s_1^\rho\ t_1^\rho) \dots (s_{r_0(\alpha)}^\rho\ t_{r_0(\alpha)}^\rho)
				(s_1\ t_1) \dots (s_{r_g(\alpha,\beta)}\ t_{r_g(\alpha,\beta)})
				(s_{r_0(\beta)}^\sigma\ t_{r_0(\beta)}^\sigma) \dots (s_1^\sigma\ t_1^\sigma)
			\end{equation}
			
		\noindent
		of length
						
			\begin{equation*}
				r_0(\alpha)+r_g(\alpha,\beta)+r_0(\beta) = 2g-2+2d=r_g(1^d,1^d),
			\end{equation*}
		
		\noindent
		which is thus a transitive loop of genus $g$.
		This loop, however, is not monotone; its spectrum
		
			\begin{equation*}
				(t_1^\rho,\dots,t_{r_0(\alpha)}^\rho,
			 	t_1,\dots, t_{r_g(\alpha,\beta)},
				t_{r_0(\beta)}^\sigma,\dots, t_1^\sigma) 
			\end{equation*}
			
		\noindent
		satisfies the inequalities
		
			\begin{align*}
				&t_1^\rho < \dots < t_{r_0(\alpha)}^\rho \\
				&t_1 \leq \dots \leq t_{r_g(\alpha,\beta)} \\
				&t_{r_0(\beta)}^\sigma > \dots > t_1^\sigma.
			\end{align*}
			
		Let $C^\rho = \{t_1^\rho, \dots, t_{r_0(\alpha)}^\rho\}$ be the set
		of colours which appear in the first part of the spectrum of the 
		transitive genus $g$ loop \eqref{eqn:TransitiveLoop}, and similarly let
		$C^\sigma = \{t_1^\sigma, \dots, t_{r_0(\beta)}^\sigma\}$ be the set
		of colours which appear in the last part of the spectrum of this loop.
		Let us use the sorting action of $\group{S}(r_g(1^d,1^d))$ 
		to rearrange the steps of \eqref{eqn:TransitiveLoop}
		in weakly increasing order:
		
			\begin{equation}
				\label{eqn:SortedLoop}
				\id = \id(s'_1\ t'_1) \dots
				(s'_{r_g(1^d,1^d))}\ t'_{r_g(1^d,1^d)}), \quad
				t'_1 \leq \dots \leq t'_{(r_g(1^d,1^d))}.
			\end{equation}
			
		\noindent
		So far, this procedure is reversible --- given the sorted loop 
		\eqref{eqn:SortedLoop} and the sets $C^\rho$ and 
		$C^\sigma$, we can reconstruct the transitive loop \eqref{eqn:TransitiveLoop},
		and hence the original walk \eqref{eqn:MonotoneWalk}.
		Thus, to establish the stated inequalities, it 
		suffices to show that, for every identity-based loop of length $2g-2+2d$
		with $c$ colours in its spectrum, there are between $3^c$ and $4^c$ pairs
		of subsets $C^\rho, C^\sigma$ for which the reverse construction succeeds.
		Consider the possibilities for each of the $c$ colours in the spectrum:
		\begin{enumerate}
			\item the colour appears in neither $C^\rho$ nor $C^\sigma$;
			\item the colour appears in $C^\rho$ but not $C^\sigma$;
			\item the colour does not appear in $C^\rho$ but appears in $C^\sigma$; or
			\item the colour appears in both $C^\rho$ and $C^\sigma$.
		\end{enumerate}
		Possibilities (1--3) are always fine for the reverse construction, whereas
		option (4) only works when the colour appears at least twice in the spectrum.
		Thus, there is an independent choice of at least 3 and at most 4 
		options for each of the $c$
		colours, resulting in between $3^c$ and $4^c$ valid pairs $C^\rho, C^\sigma$.
	\end{proof}
	
	Proposition \ref{prop:Inequalities} establishes the inequality \eqref{eqn:Inequality},
	and thus completes the proof of Theorem \ref{thm:Main}.
	
\bibliographystyle{amsplain}

\begin{thebibliography}{10}

	\bibitem{Biane}
	P. Biane,
	\emph{Parking functions of types A and B},
	Electron. J. Combin. \textbf{9} (2002), \#N7.
	
	%\bibitem{CC}
	%S. R. Carrell, G. Chapuy,
	%\textit{Simple recurrence formulas to count maps on orientable surfaces},
	%J. Combin. Theory Ser. A \textbf{133} (2015), 58-75.
	
	\bibitem{DH}
	B. De Wit, G. 't Hooft,
	\emph{Nonconvergence of the $1/N$ expansion for $SU(N)$
	gauge fields on a lattice},
	Phys. Lett. B \textbf{69}, 61-64.
	
	\bibitem{DG}
	P. Diaconis, C. Greene,
	\textit{Applications of Murphy's elements},
	Stanford University Technical Report \textbf{335} (1989).
	
	\bibitem{DDM}
	N. Do, A. Dyer, D. Matthews,
	\textit{Topological recursion and the quantum curve for monotone
	Hurwitz numbers},
	\textsf{arXiv:1408.3992}.
	
	\bibitem{DK}
	N. Do, M. Karev,
	\textit{Monotone orbifold Hurwitz numbers},
	\textsf{arXiv:1505.06503}.

	\bibitem{EEHS}
	D. Eisenbud, N. Elkies, J. Harris, R. Speiser,
	\textit{On the Hurwitz scheme and its monodromy},
	Compositio Mathematica \textbf{77} (1991), 95-117.
	
	\bibitem{GGN1}
	I. P. Goulden, M. Guay-Paquet, J. Novak,
	\textit{Monotone Hurwitz numbers in genus zero},
	Canad. J. Math. \textbf{65} (2013), 1020-1042.
	
	\bibitem{GGN2}
	I. P. Goulden, M. Guay-Paquet, J. Novak,
	\textit{Polynomiality of monotone Hurwitz numbers in higher genera},
	Adv. Math. \textbf{238} (2013), 1-23.
	
	\bibitem{GGN3}
	I. P. Goulden, M. Guay-Paquet, J. Novak,
	\textit{Monotone Hurwitz numbers and the HCIZ integral},
	Ann. Math. Blaise Pascal \textbf{21} (2014), 71-99.

	\bibitem{GJV}
	I. P. Goulden, D. M. Jackson, R. Vakil,
	\textit{Towards the geometry of double Hurwitz numbers},
	Adv. Math. \textbf{198} (2005), 43-92.
	
	\bibitem{GT}
	D. J. Gross, W. Taylor IV,
	\textit{Two-dimensional QCD is a string theory},
	Nucl. Phys. B \textbf{400} (1993), 181-208.
	
	\bibitem{Hurwitz}
	A. Hurwitz,
	\textit{Ueber Riemann'sche Fl\"achen mit gegebenen 
	Verzweigungspunkten},
	Math. Ann. \textbf{39} (1891), 1-60.
	
	\bibitem{Jucys}
	A. Jucys, 
	\textit{Symmetric polynomials and the center of the symmetric group
	ring},
	Rep. Math. Phys. \textbf{5} (1974), 107-112.

	\bibitem{MN}
	S. Matsumoto, J. Novak,
	\emph{Jucys-Murphy elements and unitary matrix integrals},
	Internat. Math. Res. Not. \textbf{2} (2013), 362-397.

	\bibitem{Novak:BCP}
	J. Novak,
	\emph{Jucys-Murphy elements and the Weingarten function},
	Banach Cent. Publ. \textbf{89} (2010), 231-235.
	
	\bibitem{Novak:JSP}
	J. Novak,
	\textit{Lozenge tilings and Hurwitz numbers},
	J. Stat. Phys. \textbf{161} (2015), 509-517.
	
	\bibitem{Okounkov}
	A. Okounkov, 
	\textit{Toda equations for Hurwitz numbers},
	Math. Res. Lett. \textbf{7} (2000), 447-453.
	
	\bibitem{OP1}
	A. Okounkov, R. Pandharipande,
	\textit{Gromov-Witten theory, Hurwitz theory, and completed cycles},
	Ann. Math. \textbf{163} (2006), 517-560.
	
	\bibitem{OP2}
	A. Okounkov, R. Pandharipande,
	\textit{Gromov-Witten theory, Hurwitz theory, and matrix models},
	\textsf{arXiv:0101147v2}.
	
	\bibitem{Pyber}
	L. Pyber, 
	\emph{Enumerating finite groups of given order},
	Ann. Math. \textbf{137} (1993), 203-220.
	
	\bibitem{Samuel}
	S. Samuel,
	\emph{$U(N)$-integrals, $1/N$, and the De Wit-'t Hooft
	anomalies},
	J. Math. Phys. \textbf{21} (1980), 2695-2703.
	
	\bibitem{Stanley}
	R. P. Stanley,
	\textit{Parking functions and noncrossing partitions},
	Electron. J. Combin. \textbf{4} (1997), \#R20.
\end{thebibliography}

\end{document}